\documentclass[12pt, leqno]{amsart}
\usepackage{graphics,amsmath,amsthm,amssymb}
\usepackage[enableskew]{youngtab}
\usepackage{tikz}
\usepackage[all,cmtip]{xy}
\usepackage{color}
\makeatletter
\newsavebox\myboxA
\newsavebox\myboxB
\newlength\mylenA

\newcommand*\xoverline[2][0.75]{%
    \sbox{\myboxA}{$\m@th#2$}%
    \setbox\myboxB\null
    \ht\myboxB=\ht\myboxA%
    \dp\myboxB=\dp\myboxA%
    \wd\myboxB=#1\wd\myboxA
    \sbox\myboxB{$\m@th\overline{\copy\myboxB}$}
    \setlength\mylenA{\the\wd\myboxA}
    \addtolength\mylenA{-\the\wd\myboxB}%
    \ifdim\wd\myboxB<\wd\myboxA%
       \rlap{\hskip 0.5\mylenA\usebox\myboxB}{\usebox\myboxA}%
    \else
        \hskip -0.5\mylenA\rlap{\usebox\myboxA}{\hskip 0.5\mylenA\usebox\myboxB}%
    \fi}
\makeatother
\newtheorem{theorem}{Theorem}[section]

\newtheorem{proposition}[theorem]{Proposition}
\newtheorem{corollary}[theorem]{Corollary}

\newtheorem{remark}[theorem]{Remark}
\newtheorem{lemma}[theorem]{Lemma}
\newtheorem{example}[theorem]{Example}

\def\vs{\vspace{.2cm}}

\begin{document}

\title[Congruences for spin characters...]{Congruences for spin characters of the double covers of the symmetric and alternating groups}
\author[R. Nath]{Rishi Nath}
\address{Department of Mathematics} 
\address{Massachusetts Institute of Technology, Cambridge MA 02139}
\address{rnath@mit.edu}
\address {York College, City University of New York, Jamaica, NY 11451}
\address{rnath@york.cuny.edu}
\author[J. A. Sellers]{James A. Sellers}
\address{Department of Mathematics, Penn State University, University Park, PA  16802}
\address {sellersj@psu.edu}


\date{\bf\today}

\begin{abstract}
Let $p$ be an odd prime. The bar partitions with sign and $p$-bar-core partitions with sign respectively label the spin characters and $p$-defect zero spin characters of the double cover of the symmetric group, and by restriction, those of the alternating group. The generating functions for these objects have been determined by J. Olsson. We study these functions from an arithmetic perspective, using classical analytic tools and elementary generating function manipulation to obtain many Ramanujan-like congruences. 
\end{abstract}

\maketitle


\bigskip


\section{Introduction}
\label{intro}
A partition of $n$ is a non-increasing sequence of positive integers that sum to $n$. Euler determined the generating function for $p(n),$ the number of partitions of $n$: 
\[
P(q):=\sum^{\infty}_{n=0}p(n)q^n=\prod_{n=1}^\infty\frac{1}{1-q^n}
\]
The partitions of $n$ label the irreducible characters of the symmetric group $S(n)$. The characters of the alternating group $A(n)$ are then labeled by restriction from $S(n).$ 
In particular, $p(n)$ gives the dimension of the character table of $S(n).$

Srinivasa Ramanujan was the first to notice several remarkable arithmetic properties of the partition function.  

\begin{theorem}
\label{ramanujan_congs}
For all $n\geq 0$,
\begin{eqnarray*}
p(5n + 4) &=& 0\pmod{5}, \\
p(7n + 5) &=& 0\pmod{7}, \\
p(11n + 6) &=& 0\pmod{11}.
\end{eqnarray*}
\end{theorem}

A $t$-core partition of $n$ is a partition in which no hook of size $t$ appears in its Young diagram. When $t$ is prime, the $t$-cores label the $t$-defect zero blocks of $S_n.$ In 1990, F. Garvan, D. Kim, and D. Stanton \cite{G-K-S} provided a beautiful proof of Theorem \ref{ramanujan_congs} using core partitions. The generating function for $t$-core partitions has itself become a rich topic of study. For example, work by F. Garvan \cite{Gar}, M. Hirschhorn and J. Sellers  \cite{H-S-1}, \cite{H-S-2}, L. Kolitsch and J. Sellers \cite{K-S} and N. Baruah and B. Berndt \cite{B-B}, has produced a multitude of Ramanujan-type congruences for $t$-core partitions.  

We now consider partitions with distinct parts. For any such partition $\lambda,$ shift the $i$th row of its Young diagram $i$ positions to the right. In this shifted diagram $S(\lambda)$ we associate a bar with a certain bar-length to each position. Then we can consider diagrams in which no bar of length $t$ occurs. Partitions of this type will be called $t$-bar-core (or $\bar{t}$-core) partitions. A. O. Morris \cite{M} proved that for $\hat{S}(n),$ the double cover of $S(n),$ bar partitions with sign label the spin characters of $\hat{S}(n)$ and, when $p$ odd, that $\bar{p}$-core partitions with sign label the $p$-defect zero spin characters of $\hat{S}(n)$.  Almost three decades later J. Humpreys \cite{H} and M. Cabanes \cite{C}, independently, proved a conjecture of Morris on the $p$-block distribution of characters of $\hat{S}(n)$. Spin characters and $p$-defect zero spin characters of $\hat{A}(n),$ the double cover of $A(n)$, are then obtained by restriction from $\hat{S}(n)$, using an application of Clifford theory. 

Recent analysis of spin characters of $\hat{S}(n)$ and $\hat{A}(n)$ has focused on their dimensionality (see \cite{BO},\cite{KT}). Here however we study congruences of generating functions for spin characters, $p$-defect zero spin characters of $\hat{S}(n)$ and $\hat{A}(n)$, and $\bar{t}$-cores. The functions for these are known, due to the work of J. Olsson.  

In section 2, we introduce the basics: partitions and bar-partitions, bar-core partitions, character theory, generating functions, and some identities. In section 3, we find a characterization of the number of spin characters of $\hat{S}(n)$ and $\hat{A}(n)$ modulo 2 and 3. We then use this to obtain infinitely many Ramanujan-like congruences modulo 2 and 3. In section 4, we prove a number of parity results for the number of $\bar{p}$-cores of $n$, as well as Ramanujan--like congruences for $p$-defect zero spin characters of $\hat{S}(n)$ and $\hat{A}(n)$.

\section{Preliminaries}
\label{sec:recurs}
\subsection{Partitions and bar partitions}
Let $\lambda$ be a partition of $n$ into $k$ parts.  We form  $[\lambda],$ the Young diagram of $\lambda,$ by stacking left-aligned rows of boxes, with the $r$th row having $a_r$ boxes, as $1\leq k$.  The $j$th box in the $i$th row will be in $(i,j)$-position.  The hook $h_{ij}$ associated to a box in $(i,j)$-position will be the union of the set of boxes in the $i$th row and to left of $(i,j)$, those in the $j$th column and below $(i,j)$, and the box $(i,j)$ itself. The length of $h_{ij}$ will be the total number of boxes in the set.  A $t$-core partition $\lambda$ is one in which no hooks of size $t$ appear in $[\lambda]$.

We say $\lambda$ is a bar partition of $n$ if $\lambda=(a_1,\cdots,a_k)$ and $a_1>a_2>\cdots>a_k$; that is, a partition with distinct parts. The set of bar partitions of $n$ is denoted $\hat{P}(n)$ and $|\hat{P}(n)|=q(n).$ Let $S(\lambda)$ be the shifted Young diagram of $\lambda.$ It is obtained from the usual Young diagram by shifting the $i$-th row $(i-1)$ positions to the right. 

As above, the $j$th box in the $i$th row will be in $(i,j)$-position. We denote the $(i,j)$-bar length as $\bar{h}_{ij}(\lambda)$. To each node $(i,j)$ in $S(\lambda)$ we associate a bar length as follows: the bar lengths in the $i$-th row are obtained by writing the following sets in decreasing order:
\[
{\mathcal H}_i(\lambda)=\{1,2,\cdots,a_i\}\cup\{a_i+a_j| j>i\}\backslash\{a_i-a_j|j>i\}
\]
We let $\hat{\mathcal H}=\cup_i{\mathcal H}_i(\lambda),$ a multiset. We denote the $(i,j)$-bar length by $\bar{h}_{ij}$. 
\begin{example}
Let $\lambda=(5,3,2).$ Then ${\mathcal H}_1=\{1,2,3,4,5\}\cup\{5+2,5+3\}\backslash{\{5-3,5-2\}}=\{1,2,4,5,7,8\}$ ${\mathcal H}_2=\{1,2,3\}\cup\{3+2\}\backslash\{3-2\}=\{2,3,5\}$ and ${\mathcal H}_3=\{1,2\}.$ 
See Figure 1.
\end{example}
\begin{figure}[h!]\label{x}
\caption{$S(\lambda)$ with bar lengths for $\lambda=(5,3,2)$}
\young(87541,:532,::21)
\end{figure}
\subsection{Bar-core partitions and the bar-abacus}
Let $p$ be an odd prime. Bar partitions $\lambda$ in which $p$ does not appear as bar length in $S(\lambda)$ are called $\bar{p}$-core (or $p$-bar-core) partitions.

Let the $p$-abacus consist of $p$-runners going from south-to-north numbered $0,1,\cdots,p-1.$ On each runner there are positions numbered $0,1,2,\cdots$. For $0\leq i\leq p-1$ we define 
$X^{\lambda}_i=\{a\in {\mathbb N}|$ there exists a $k$ such that $1\leq k\leq m$ and $a_k=ap+i.\}$
Then $p$-abaci of $\bar{p}$-core partitions can be described as follows:
\begin{enumerate}
\item There are no beads on the 0th runner
\item For the $1\leq j\leq p$, if $|X^{\lambda}_j|=k>0$ then there are beads in the first $k$ positions on the $j$th runner.
\item If $|X^{\lambda}_j|=k>0$ then $|X_{p-j}|=0$.
\end{enumerate}
For details on the abaci of bar-cores see \cite[Chapter 4]{O3}.
\subsection{Character theory}
The irreducible characters of $S(n)$ are labeled by the partitions of $n$. Let $p$ be a prime. Then the $p$-core partitions of $n$, that is, those in which no hook of size $p$ appear, label the $p$-defect zero blocks of the symmetric group $S(n).$

In 1911, Schur proved that the symmetric groups $S(n)$ have covering groups $\hat{S}(n)$ of order $2n!.$ Thus there is a non-split exact sequence 
\[
\begin{array}{ccccccccc} 
1 & \to & \langle z \rangle & \to & \hat{S}(n) & \to & S(n) & \to & 1\\
\end{array}
\]
where $\langle z\rangle$ is a central subgroup of order $2$ in $\hat{S}(n).$ 

The irreducible characters of $\hat{S}(n)$ which have $\langle z\rangle$ in their kernel will be referred to as {\it ordinary} characters of $\hat{S}(n).$ They are the complex irreducible characters of $S(n)\simeq \hat{S}(n)/\langle z\rangle$ labeled by the partitions of $n$. The other irreducible characters of $\hat{S}(n)$ are called spin characters, labeled canonically by the bar partitions of $n$. 

The set of bar partitions is divided into
\[
\hat{\mathcal P}^+(n)\;=\;\{\lambda\in \hat{\mathcal P}(n) | n -l(\lambda) \mbox{ even} \}
\]
\[
\hat{\mathcal P}^-(n)\;=\;\{\lambda\in \hat{\mathcal P}(n) | n -l(\lambda) \mbox{ odd} \}
\]
where $l(\lambda)$ is the length of $\lambda.$ Then each $\lambda \in {\hat{\mathcal P}}^+(n)$ labels a self-associate spin character of $\langle \lambda \rangle$ of $\hat{S}_n$ and each $\lambda \in {\hat{\mathcal P}}^-(n)$ labels a pair of non-self-associate irreducible spin characters $\langle \lambda \rangle$ and $\langle \lambda \rangle '$. 

\begin{example} \label{first_ex} Consider $\hat{S}_7.$ Then there are six non--associate irreducible spin characters ${\langle \lambda_i \rangle}, {\langle {\lambda}_i \rangle}'$ for $1\leq i\leq 3$ where $\lambda_1=(6,1), \lambda_2=(5,2), \lambda_3=(4,3)$. There are two self-associate irreducible spin characters ${\langle \lambda_4 \rangle}$ and ${\langle \lambda_5 \rangle}$ where $\lambda_4=(7)$ and $\lambda_5=(4,2,1).$
\end{example}
Let $A(n)$ be the alternating group on $n$ letters. Since the index $[S(n):A(n)]=2,$ an application of Clifford theory yields the following: two irreducible characters of $S(n)$ labeled by conjugate partitions restrict to the same irreducible character of $A(n)$; those labeled by self-conjugate partitions split into two distinct irreducible characters of $A(n).$ A similar restriction applies to the $p$-defect zero blocks of $A(n)$. 

From a similar application of Clifford theory, we obtain the character theory of $\hat{A}(n)$ from $\hat{S}(n)$ by restriction; that is a pair of non-self-associate spin characters of $\hat{S}(n)$ restrict to one of $\hat{A}(n)$ while a self-associate character of $\hat{S}(n)$ splits to two conjugate irreducible representations of $\hat{A}(n).$  This duality leads to the similar generating functions seen in Sections 2.4 and 2.5, and pairs of results in Sections 3.1 and 3.2; and 4.2 and 4.3.

When $\lambda\in \hat{\mathcal P}(n)$ then $\bar{f}_{\lambda}$ denotes the degree of the spin character(s) labelled by $\lambda$. Schur proved that if $\lambda=(a_1,a_2,\cdots,a_m)\in\hat{P}(n)$
then
\[
\bar{f}_{\lambda}=2^{\lfloor\frac{n-m}{2}\rfloor}\frac{n!}{\prod_i a_i}\prod_{i<j}\frac{a_i-a_j}{a_i+a_j}
\]
Then, by the definition of $\hat{\mathcal H}(\lambda)$ 
\[
\bar{f}_{\lambda}=2^{\lfloor\frac{n-m}{2}\rfloor}\frac{n!}{\prod_{h\in \hat{\mathcal H}(\lambda)}h}.
\]
When $p$ is an odd prime, we calculate $\bar{d}_p(\lambda)$ the $p$-defect of a spin character labeled by $\lambda$, using the following formula:
\[
\bar{d}_p(\lambda)=\nu_p(n!)-\nu_p(\prod_{h\in\hat{\mathcal H}(\lambda)}h)
\]
where $\nu_p$ is the $p$-adic valuation. In particular, the $p$-defect zero spin characters of $\hat{S}(n)$ are labeled precisely by $\bar{p}$-core partitions of $n$. To enumerate them we consider the number of $\bar{p}$-core partitions of $n$ in $\hat{\mathcal P}^+(n)$ plus twice the number of $\bar{p}$-core partitions in  $\hat{\mathcal P}^-(n).$ We will use this in the next section.

\subsection{Generating functions for $\hat{S}(n)$}
Throughout the remainder of this paper, we will use a notation for infinite products which is commonly known as the Pochhammer symbol.  This symbol is defined as follows:  

\begin{eqnarray*}
(a;q)_0 &=& 1, \\
(a;q)_n &=& (1-a)(1-aq)\dots (1-aq^{n-1}), \text{\ \ and}\\
(a;q)_\infty &=& \lim_{n\to\infty} (a;q)_n
\end{eqnarray*}
This notation is commonly used by members of the $q$--series community and allows for generating functions to be written in a compact form.  

The following propositions appear as equations 9.1, 9.5, 9.9(iii), 9.9(v) and 9.10 in Olsson \cite{O3} (see also \cite{MO}, \cite{O1}, and \cite{O2}).
\begin{proposition} The generating function for the number of irreducible characters of $S(n)$ is
\[
P(q)=\frac{1}{(q;q)_\infty}.
\]
\end{proposition}
\noindent 
(As noted above, this is simply the generating function for $p(n),$ the number of integer partitions of $n.$)  

\begin{proposition} 
\label{genfn_spchar_hatS}
Let $f_{\hat{S}}(n)$ be the number of irreducible spin characters of $\hat{S}(n)$.  Then the generating function for $f_{\hat{S}}(n)$ is
\[
\sum_{n=0}^\infty f_{\hat{S}}(n)q^n = \hat{P}(q) =\frac{(q^2;q^2)_\infty}{(q;q)_\infty}\left(\frac{3}{2}-\frac{1}{2}\frac{(q^2;q^2)_\infty^2}{(q^4;q^4)_\infty}\right).
\]
\end{proposition}
\begin{proposition} 
\label{genfn_spchar_barp}
Let $f_{\bar{p}}(n)$ be the number of $\bar{p}$-core partitions of $n.$ Then the generating function for $f_{\bar{p}}(n)$ is
\[
\sum_{n=0}^\infty f_{\bar{p}}(n)q^n = F_{\bar{p}}(q)=\frac{(q^2;q^2)_\infty(q^p;q^p)_\infty^{\frac{p+1}{2}}}{(q;q)_\infty(q^{2p};q^{2p})_\infty}.
\]
\end{proposition}
\begin{proposition} Let $f^+_{\bar{p}}(n)$ and $f^-_{\bar{p}}(n)$ be the number of $\bar{p}$-core partitions of $n$ with positive and negative sign, respectively.
Then the generating function for $f^{\pm}_{\bar{p}}(n)$ is 
\[
F^{\pm}_{\bar{p}}(x)=\frac{1}{2}(F_{\bar{p}}(x)\pm P(-x)P(-x^p)^{1-t}).
\]
\end{proposition}
\begin{proposition}
\label{genfn_pdefectzero_S} Let $f^{0}_{\hat{S},p}(n)$ be the number of $p$-defect zero spin characters of $\hat{S}(n).$  Then the generating function for $f^{0}_{\hat{S},p}(n)$ is
\[
\sum_{n=0}^\infty f^{0}_{\hat{S},p}(n)q^n = \hat{F}_{\bar{p}}(q)= F^+_{\bar{p}}(q) + 2F^-_{\bar{p}}(q).
\]
\end{proposition}

\subsection{Generating functions for $\hat{A}(n)$}
The following propositions appear as equations 9.11 and 9.12 in Olsson \cite{O3}. 
\begin{proposition} 
\label{genfn_spchar_hatA}
Let $f_{\hat{A}}(n)$ be the number of irreducible spin characters of $\hat{A}(n)$.  Then the generating function for $f_{\hat{A}}(n)$ is
\[
\sum_{n=0}^\infty f_{\hat{A}}(n)q^n =
\hat{\hat{P}}(q)=\frac{(q^2;q^2)_\infty}{(q;q)_\infty}\left(\frac{3}{2}+\frac{1}{2}\frac{(q^2;q^2)_\infty^2}{(q^4;q^4)_\infty}\right).
\]
\end{proposition}

\begin{proposition} 
\label{genfn_pdefectzero_A}Let $f^{0}_{\hat{A},p}(n)$ be the number of $p$-defect zero spin characters of $\hat{A}(n).$  Then the generating function for $f^{0}_{\hat{A},p}(n)$ is
\[
\sum_{n=0}^\infty f^{0}_{\hat{A},p}(n)q^n = \hat{\hat{F}}_{\bar{p}}(q)= 2F^+_{\bar{p}}(q) + F^-_{\bar{p}}(q).
\]
\end{proposition}

\subsection{Generating Function Manipulation Tools}
In order to prove the arithmetic properties which will be outlined below, we will utilize a few classical product--to--sum results.  We highlight those results here.  

\begin{lemma}
\label{JacobiTPI}
(Jacobi's Triple Product Identity)
For $z\neq 0$ and $\vert \,q\,\vert < 1,$ 
$$
\sum_{n=-\infty}^\infty z^nq^{n^2} = \prod_{n=0}^\infty (1-q^{2n+2})(1+zq^{2n+1})(1+z^{-1}q^{2n+1}).
$$
\end{lemma}
\begin{proof}
See \cite[Theorem 2.8]{AndrBook}.  
\end{proof}

An extremely important corollary of Lemma \ref{JacobiTPI} is commonly known as Euler's Pentagonal Number Theorem and is worth highlighting here.  

\begin{corollary}
\label{EulerPNT}
(Euler's Pentagonal Number Theorem)
$$
\sum_{n=-\infty}^\infty (-1)^nq^{n(3n-1)/2} = \prod_{n=0}^\infty (1-q^{n}).
$$
\end{corollary}

\begin{lemma}
\label{QPI}
(The Quintuple Product Identity)
For $t\neq 0,$ 
$$
\sum_{n= -\infty}^\infty s^{(3n^2+n)/2}(t^{3n} - t^{-3n-1}) = \prod_{n\geq 1} (1-s^n)(1-s^n t)(1-s^{n-1}t^{-1})(1-s^{2n-1}t^2)(1-s^{2n-1}t^{-2}).
$$
\end{lemma}
\begin{proof}
For a proof of this result, as well as alternate forms of the Quintuple Product Identity, see \cite[pp.18--19]{Berndt}.  
\end{proof}

\section{Congruences for spin characters}
\subsection{Spin characters of $\hat{S}_n$}
\vs
We now consider arithmetic properties of the function $f_{\hat{S}}(n).$ Elementary generating function manipulation techniques, outlined below, allow us to prove congruences modulo both 2 and 3.  

We begin by characterizing $f_{\hat{S}}(n)$ modulo 2 based solely on $n.$
\begin{theorem}
\label{spchar_hatS_mod2}
For all $n\geq 1,$ 
$$f_{\hat{S}}(n) \equiv R(n) \pmod{2}$$
where $R(n)$ is the number of ways to represent $n$ as 
$$n = \left(\frac{3}{2}m^2 + \frac{1}{2}m\right) + 2k^2$$
where $m$ is an integer and $k$ is a nonnegative integer.  
\end{theorem}
\begin{proof}
We begin with  the generating function result from Proposition \ref{genfn_spchar_hatS}.  Note that 
\allowdisplaybreaks
\begin{eqnarray*}
\sum_{n=0}^\infty f_{\hat{S}}(n)q^n 
&=& 
\frac{(q^2;q^2)_\infty}{(q;q)_\infty}\left(\frac{3}{2}-\frac{1}{2}\frac{(q^2;q^2)_\infty^2}{(q^4;q^4)_\infty}\right)\\
&=& 
(-q;q)_\infty\left(\frac{3}{2}-\frac{1}{2}\frac{(q^2;q^2)_\infty}{(-q^2;q^2)_\infty}\right) \\
&=& 
(-q;q)_\infty\left(\frac{3}{2}-\frac{1}{2}\sum_{k=-\infty}^\infty (-1)^kq^{2k^2}\right) \\
&&\qquad \qquad \text{using Jacobi's Triple Product Identity} \\
&=& 
(-q;q)_\infty\left(1-\sum_{k=1}^\infty (-1)^kq^{2k^2}\right) \\
&=& 
(-q;q)_\infty\left(\sum_{k=0}^\infty (-1)^kq^{2k^2}\right) \\
&\equiv & 
(q;q)_\infty\left(\sum_{k=0}^\infty q^{2k^2}\right) \pmod{2} \\
&\equiv & 
\sum_{m=-\infty}^\infty q^{m(3m+1)/2} \left(\sum_{k=0}^\infty q^{2k^2}\right) \pmod{2}.\\
&&\qquad \qquad \text{by Euler's Pentagonal Number Theorem} 
\end{eqnarray*}
The result follows by comparing the coefficient of $q^n$ on both sides of this congruence.  
\end{proof}
\begin{example} Let $n$=7. Then only $(2,0)$ and $(-2,1)$ satisfy the definition of $R(7)$. By Example \ref{first_ex}, $|\hat{\mathcal P}^+(n)|=2$ and $|\hat{\mathcal P}^-(n)|=2$. Since each element of ${\mathcal P}^-(n)$ labels two non-self-associate spin characters of $\hat{S}(7)$, we have $f_{\hat{S}}(7)=6.$ Hence $f_{\hat{S}}(7)\equiv|R(7)|\equiv 0\pmod{2}.$
\end{example}
Theorem \ref{spchar_hatS_mod2} can be used to write down infinitely many Ramanujan--like congruences modulo 2 (similar to the congruences mentioned above in Theorem \ref{ramanujan_congs}).   
\begin{corollary}
\label{Smod2_cor}
Let $p$ be a prime congruent to 5 or 11 modulo 24, and let $r$ be an integer satisfying $1\leq r\leq p-1.$  Then, for all $n\geq 0,$ 
$$
f_{\hat{S}}\left(p^2n+pr+\frac{p^2-1}{24}\right) \equiv  0 \pmod{2}.
$$
\end{corollary}
\begin{proof}
Note that $p^2n+pr+\frac{p^2-1}{24} $ can be represented as 
$$p^2n+pr+\frac{p^2-1}{24}  = \left(\frac{3}{2}m^2 + \frac{1}{2}m\right) + 2k^2$$
if and only if 
$24\left(p^2n+pr+\frac{p^2-1}{24} \right) + 1$ can be represented as 
$$24\left(p^2n+pr+\frac{p^2-1}{24}\right) + 1 = (6m+1)^2 + 3(4k)^2.$$
Since $p$ is a prime congruent to 5 or 11 modulo 24, we know that $\left(\frac{-3}{p}\right) = -1$ where $\left(\frac{a}{p}\right)$ is the Legendre symbol.  Thus, if 
$$
24\left(p^2n+pr+\frac{p^2-1}{24} \right) + 1
$$
can be written as $(6m+1)^2 + 3(4k)^2,$ then it must be the case that 
$$\nu_p\left(24\left(p^2n+pr+\frac{p^2-1}{24} \right) + 1\right)$$ is even (where $\nu_p(n)$ is the exponent of $p$ dividing $n.$)  However, 
\begin{eqnarray*}
24\left(p^2n+pr+\frac{p^2-1}{24} \right) + 1
&=& 
24p^2n + 24pr + p^2 \\
&=& 
p(24pn + 24r + p).
\end{eqnarray*}
This quantity is clearly divisible by $p$ and  {\bf not} divisible by $p^2.$  Therefore, 
$$\nu_p\left(24\left(p^2n+pr+\frac{p^2-1}{24} \right) + 1\right)$$ is actually odd, and this means that 
$24\left(p^2n+pr+\frac{p^2-1}{24} \right) + 1$ {\bf cannot} be represented as 
$$24\left(p^2n+pr+\frac{p^2-1}{24} \right) + 1 = (6m+1)^2 + 3(4k)^2.$$
Therefore, by Theorem \ref{spchar_hatS_mod2}, the result follows.  
\end{proof}
Thus, for example, Corollary \ref{Smod2_cor} yields the following when $p=5:$  
\begin{eqnarray*}
f_{\hat{S}}(25n+6) &\equiv & 0 \pmod{2}, \\
f_{\hat{S}}(25n+11) &\equiv & 0 \pmod{2}, \\
f_{\hat{S}}(25n+16) &\equiv & 0 \pmod{2}, \text{\ \ and} \\
f_{\hat{S}}(25n+21) &\equiv & 0 \pmod{2}.
\end{eqnarray*}
We note that a very similar argument to the one used above was recently used by Chen, Hirschhorn, and Sellers \cite[Corollary 4.6]{CHS}.

We next turn our attention to a consideration of $f_{\hat{S}}(n)$ modulo 3 and a characterization based solely on $n.$  

\begin{theorem}
\label{spchar_hatS_mod3}
For all $n\geq 1,$ 
$$f_{\hat{S}}(n) \equiv 
\begin{cases} 
1 \pmod{3} &\mbox{if } n = \frac{k(3k+1)}{2}, k\equiv 0, 3 \pmod{4}, \\ 
2 \pmod{3} & \mbox{if } n  = \frac{k(3k+1)}{2}, k\equiv 1, 2 \pmod{4}, \\ 
0 \pmod{3} & \mbox{otherwise}.
\end{cases} 
$$
\end{theorem}
\begin{proof}
As in the proof of Theorem \ref{spchar_hatS_mod2}, we begin with  the generating function result from Proposition \ref{genfn_spchar_hatS}.  Note that 
\begin{eqnarray*}
2\sum_{n=0}^\infty f_{\hat{S}}(n)q^n 
&=& 
\frac{(q^2;q^2)_\infty}{(q;q)_\infty}\left({3}-\frac{(q^2;q^2)_\infty^2}{(q^4;q^4)_\infty}\right)\\
&\equiv & 
\frac{(q^2;q^2)_\infty}{(q;q)_\infty}\left(2\frac{(q^2;q^2)_\infty^2}{(q^4;q^4)_\infty}\right) \pmod{3}\\
&=& 
2\frac{(q^2;q^2)_\infty^3}{(q;q)_\infty (q^4;q^4)_\infty}.  
\end{eqnarray*}
We now wish to rewrite this last product as a sum, and in order to do so we utilize the Quintuple Product Identity from Lemma \ref{QPI}.  Performing the substitutions $s=-q$ and $t=-1$ in Lemma \ref{QPI}, the product side becomes 
\allowdisplaybreaks 
\begin{eqnarray*}
&&
\prod_{n\geq 1} (1-(-q)^n)(1-(-q)^n (-1))(1-(-q)^{n-1}(-1)^{-1})(1-(-q)^{2n-1}(-1)^2)(1-(-q)^{2n-1}(-1)^{-2}) \\
&=& 
\prod_{n\geq 1} (1-(-q)^n)(1+(-q)^n)(1+(-q)^{n-1})(1+q^{2n-1})^2 \\
&=& 
2\prod_{n\geq 1} (1-(-q)^n)(1+(-q)^n)^2(1+q^{2n-1})^2 \\
&=& 
2\prod_{n\geq 1} (1-q^{2n})(1+q^{2n-1})(1+q^{2n})^2(1-q^{2n-1})^2(1+q^{2n-1})^2 \\
&=& 
2\prod_{n\geq 1} (1-q^{n})(1-q^{4n-2})(1+q^{n})^2 \\
&=& 
2\prod_{n\geq 1} (1-q^{n})\frac{(1-q^{2n})}{(1-q^{4n})}\frac{(1-q^{2n})^2}{(1-q^{n})^2} \\
&=& 
2\prod_{n\geq 1} \frac{(1-q^{2n})^3}{(1-q^{4n})(1-q^{n})}\\
&=& 
2\frac{(q^2;q^2)_\infty^3}{(q;q)_\infty(q^4;q^4)_\infty}.
\end{eqnarray*}
Note that this was the product we obtained above which is congruent, modulo 3, to the generating function for $2f_{\hat{S}}(n).$  Thus, we now know that the generating function for $2f_{\hat{S}}(n)$ is congruent modulo 3 to the sum side obtained from the Quintuple Product Identity when we make the substitutions $s=-q$ and $t=-1.$  In this case, the sum side becomes 
\begin{eqnarray*}
&&
\sum_{n= -\infty}^\infty (-q)^{(3n^2+n)/2}((-1)^{3n} - (-1)^{-3n-1}) \\
&=& 
\sum_{n= -\infty}^\infty q^{(3n^2+n)/2}(-1)^{(3n^2+n)/2}((-1)^{3n} - (-1)^{3n+1}) \\
&=& 
\sum_{n= -\infty}^\infty q^{(3n^2+n)/2}\left((-1)^{(3n^2+n)/2 + 3n}- (-1)^{(3n^2+n)/2 + 3n+1}   \right)  \\
&=& 
\sum_{n= -\infty}^\infty q^{(3n^2+n)/2}\left((-1)^{(3n^2+n)/2 + 3n}+ (-1)^{(3n^2+n)/2 + 3n}   \right)  \\
&=& 
2\sum_{n= -\infty}^\infty q^{(3n^2+n)/2}\left((-1)^{(3n^2+n)/2 + 3n}  \right)  \\
&=& 
2\sum_{n= -\infty}^\infty q^{(3n^2+n)/2}\left((-1)^{(3n^2+7n)/2 }  \right) .
\end{eqnarray*}
The result now follows by comparing the coefficients of like powers of $q$ in the above work and doing a bit of analysis on the function $(-1)^{(3n^2+7n)/2 }. $

\end{proof}
Theorem \ref{spchar_hatS_mod3} provides a powerful tool for proving infinitely many Ramanujan--like congruences modulo 3 which are satisfied by  $f_{\hat{S}}(n).$

\begin{corollary}
\label{spchar_hatS_mod3_infinitefamily}
Let $p\geq 5$ be prime and choose $r,$ $1\leq r\leq p-1,$ such that $24r+1$ is a quadratic nonresidue modulo $p.$  Then, for all $n\geq 0,$ 
$$
f_{\hat{S}}(pn+r)\equiv 0\pmod{3}.  
$$
\end{corollary}

\begin{proof}
Given a prime $p$ and a value of $r$ as described in the corollary, we must ask whether $pn+r$ can be represented as a pentagonal number $k(3k+1)/2$ for some $k.$  Note that 
\begin{eqnarray*}
&&
pn+r = \frac{k(3k+1)}{2} \\
&\iff &
24(pn+r)+1 = 36k^2+12k+1 \\
&\iff &
24pn+24r+1 = (6k+1)^2.
\end{eqnarray*}
Modulo $p,$ this means $24r+1$ must be equivalent to a square.  But this cannot be since $r$ was chosen such that $24r+1$ is a quadratic nonresidue modulo $p.$  Thus, $pn+r$ cannot be represented as $k(3k+1)/2$ for some $k.$  Therefore, by Theorem \ref{spchar_hatS_mod3}, 
$$
f_{\hat{S}}(pn+r)\equiv 0\pmod{3}.  
$$
\end{proof}
We note that, for each prime $p\geq 5,$ the above corollary provides $(p-1)/2$ different Ramanujan--like congruences modulo 3 which are satisfied by $f_{\hat{S}}.$    
\vs

\subsection{Spin characters of $\hat{A}(n)$}
\vs
The duality between the generating functions in Propositions \ref{genfn_spchar_hatS} and \ref{genfn_spchar_hatA} results in a characterization of $f_{\hat{A}}(n)$ modulo 2 which is almost identical to that of $f_{\hat{S}}(n).$

\begin{theorem}
\label{spchar_hatA_mod2}
For all $n\geq 1,$ 
$$f_{\hat{A}}(n) \equiv R(n) \pmod{2}$$
where $R(n)$ is the number of ways to represent $n$ as 
$$n = \left(\frac{3}{2}m^2 + \frac{1}{2}m\right) + 2k^2$$
where $m$ is an integer and $k$ is a positive integer.  
\end{theorem}
\begin{remark}
It is important to highlight that, while this result looks extremely similar to Theorem \ref{spchar_hatS_mod2}, there is a subtle difference.  In Theorem \ref{spchar_hatS_mod2}, $k=0$ is allowed while in Theorem \ref{spchar_hatA_mod2}, $k$ is required to be positive.  
\end{remark}
\begin{proof}
The proof of this result is almost identical to that of Theorem \ref{spchar_hatS_mod2} and is, therefore, omitted here.  
\end{proof}
As with Theorem \ref{spchar_hatS_mod2}, we can write down an infinite family of Ramanujan--like congruences modulo 2 satisfied by $f_{\hat{A}}.$  Such a corollary looks almost identical to Corollary \ref{Smod2_cor}.
\begin{corollary}
\label{Amod2_cor}
Let $p$ be a prime congruent to 5 or 11 modulo 24, and let $r$ be an integer satisfying $1\leq r\leq p-1.$  Then, for all $n\geq 0,$ 
$$
f_{\hat{A}}\left(p^2n+pr+\frac{p^2-1}{24}\right) \equiv  0 \pmod{2}.
$$
\end{corollary}

We can also easily write down an analogue of Theorem \ref{spchar_hatS_mod3} for the function $f_{\hat{A}}(n).$  
\begin{theorem}
\label{spchar_hatA_mod3}
For all $n\geq 1,$ 
$$f_{\hat{A}}(n) \equiv 
\begin{cases} 
2 \pmod{3} &\mbox{if } n = k(3k+1)/2, k\equiv 0, 3 \pmod{4}, \\ 
1 \pmod{3} & \mbox{if } n  = k(3k+1)/2, k\equiv 1, 2 \pmod{4}, \\ 
0 \pmod{3}  & \mbox{otherwise}.
\end{cases} 
$$
\end{theorem}
\begin{proof}
As noted in Proposition \ref{genfn_spchar_hatA}, the generating function for $f_{\hat{A}}(n)$ is given by 
\[
\sum_{n=0}^\infty f_{\hat{A}}(n)q^n =
\hat{\hat{P}}(q)=\frac{(q^2;q^2)_\infty}{(q;q)_\infty}\left(\frac{3}{2}+\frac{1}{2}\frac{(q^2;q^2)_\infty^2}{(q^4;q^4)_\infty}\right).
\]
Thus, 
\begin{eqnarray*}
2\sum_{n=0}^\infty f_{\hat{A}}(n)q^n 
&=& 
\frac{(q^2;q^2)_\infty}{(q;q)_\infty}\left({3}+\frac{(q^2;q^2)_\infty^2}{(q^4;q^4)_\infty}\right)\\
&\equiv & 
\frac{(q^2;q^2)_\infty}{(q;q)_\infty}\left(\frac{(q^2;q^2)_\infty^2}{(q^4;q^4)_\infty}\right) \pmod{3}\\
&=& 
\frac{(q^2;q^2)_\infty^3}{(q;q)_\infty (q^4;q^4)_\infty}.  
\end{eqnarray*}
Therefore, 
\begin{eqnarray*}
\sum_{n=0}^\infty f_{\hat{A}}(n)q^n 
& \equiv & 
2\frac{(q^2;q^2)_\infty^3}{(q;q)_\infty (q^4;q^4)_\infty}\pmod{3} \\
&\equiv & 
2\sum_{n=0}^\infty f_{\hat{S}}(n)q^n \pmod{3}
\end{eqnarray*}
thanks to the proof of Theorem \ref{spchar_hatS_mod3}.  The result then follows from Theorem \ref{spchar_hatS_mod3}.
\end{proof}
Because of the extremely similar look of the results in Theorems \ref{spchar_hatS_mod3} and \ref{spchar_hatA_mod3}, it is the case that a corollary to Theorem  \ref{spchar_hatA_mod3} can be written down immediately which closely resembles Corollary \ref{spchar_hatS_mod3_infinitefamily}.

\begin{corollary}
\label{spchar_hatA_mod3_infinitefamily}
Let $p\geq 5$ be prime and choose $r,$ $1\leq r\leq p-1,$ such that $24r+1$ is a quadratic nonresidue modulo $p.$  Then, for all $n\geq 0,$ 
$$
f_{\hat{A}}(pn+r)\equiv 0\pmod{3}.  
$$
\end{corollary}
\vs
\section{Arithmetic Results on $p$-bar core partitions}
\subsection{Ramanujan--like congruences for $\bar{p}$-core partitions}
Let $p$ be an odd prime. Using elementary generating function manipulations, we can easily prove a number of parity results for $f_{\bar{p}}(n),$ the number of $\bar{p}$-core partitions of $n.$ 
\begin{theorem}
\label{pbar-cores-parity}
Let $p\geq 5$ be prime and let $r$, $1\leq r\leq p-1,$ such that $24r+1$ is a quadratic nonresidue modulo $p.$  Then, for all $n\geq 0,$ 
$$
f_{\bar{p}}(pn+r) \equiv 0\pmod{2}.
$$
\end{theorem}
\begin{proof}
Note that 
\begin{eqnarray*}
\sum_{n=0}^\infty f_{\bar{p}}(n)q^n
&=& 
\frac{(q^2;q^2)_\infty(q^p;q^p)_\infty^{\frac{p+1}{2}}}{(q;q)_\infty(q^{2p};q^{2p})_\infty} \\
&\equiv &
\frac{(q;q)_\infty^2(q^p;q^p)_\infty^{\frac{p+1}{2}}}{(q;q)_\infty(q^{p};q^{p})_\infty^2} \pmod{2} \\
&=& 
(q;q)_\infty(q^p;q^p)_\infty^{\frac{p+1}{2}-2} \\
&=& 
(q;q)_\infty(q^p;q^p)_\infty^{\frac{p-3}{2}} \\
&\equiv & 
\sum_{m=-\infty}^\infty q^{m(3m+1)/2} (q^p;q^p)_\infty^{\frac{p-3}{2}} \pmod{2}.
\end{eqnarray*}
Note that $  (q^p;q^p)_\infty^{\frac{p-3}{2}} $ is a function of $q^p,$ and since we are only concerned about arithmetic progressions of the form $pn+r$ where $1\leq r\leq p-1,$ it is the case that we can ignore the factor of 
$  (q^p;q^p)_\infty^{\frac{p-3}{2}} $ as we move forward.  

Now we simply need to ask:  Can $pn+r$ ever be written as $m(3m+1)/2$ for some integer $m$?  Note that 
$$ 
pn+r = \frac{m(3m+1)}{2} 
\iff 
24(pn+r) +1 = (6m+1)^2.
$$
Thus, $24r+1\equiv (6m+1)^2 \pmod{p}.$  However, $r$ has been chosen such that $24r+1$ is a quadratic nonresidue modulo $p,$ so we know that it cannot be congruent to a square modulo $p.$  Therefore, $pn+r$ cannot be written as $m(3m+1)/2$ for any integer $m,$ and this means that $f_{\bar{p}}(pn+r)\equiv 0 \pmod{2}$ based on the generating function manipulations above.  
\end{proof}
The following example illustrates Theorem \ref{pbar-cores-parity}.
\begin{example} \label{yes} Let $p=7$, $n=4,$ and $r=3.$  Then $pn+r=31$ is quadratic nonresidue mod 7. Let ${\mathbb F}_{\bar 7}(31)$ be the set of $\bar{7}$-cores of 31 so that $|{\mathbb F}_{\bar 7}(31)|=f_{\bar{7}}(31)$. Since
\[
{\mathbb F}_{\bar 7}(31)=\{(16,9,5,2),(12,10,5,3,1),(17,10,3,1),(16,9,3,2,1)\},
\]
we have $f_{\bar{7}}(31)\equiv 0 \pmod{2}$.  See Figures 2,3,4 and 5 for the corresponding 7-abaci arrangements. 
\end{example}
\begin{figure}[h]
\begin{center}
{\rotatebox{0}{\resizebox*{5cm}{!}{\includegraphics{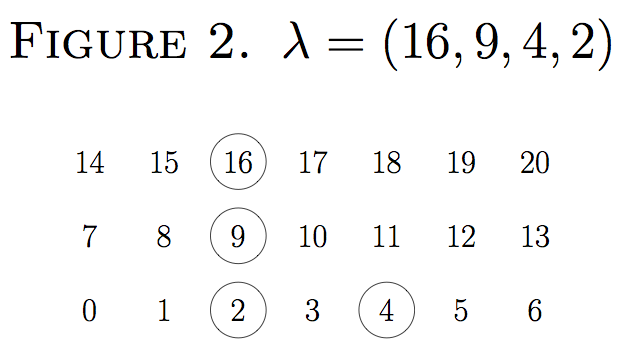}}\hglue5mm}}
{\rotatebox{0}{\resizebox*{5.5cm}{!}{\includegraphics{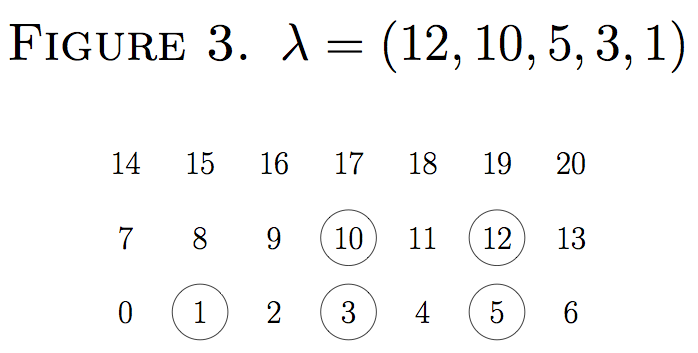}}\hglue5mm}}
\label{figs1+2}
\end{center}
\end{figure}

\begin{figure}[h]
\begin{center}
{\rotatebox{0}{\resizebox*{5.2cm}{!}{\includegraphics{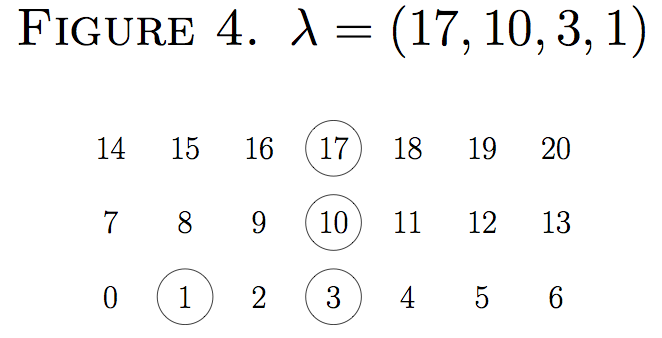}}\hglue5mm}}
{\rotatebox{0}{\resizebox*{5.5cm}{!}{\includegraphics{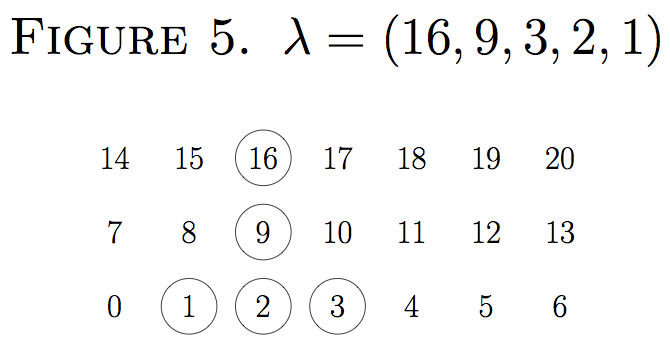}}\hglue5mm}}
\label{figs1+2}
\end{center}
\end{figure}
We note that our results in this section are slightly strengthened by a theorem of I. Kiming \cite{K2}, that is, $f_{\bar{p}}(n)>0$ for $n\in {\mathbf N}$ when $p\geq 7$ is an odd prime. 
\subsection{$p$-defect-zero spin characters of $\hat{S}_n$}
Using Proposition \ref{genfn_pdefectzero_S}, we can now prove that $f^{0}_{\hat{S},p}$ satisfies infinitely many congruences modulo 3.  
\begin{theorem}
\label{pdefectzero_S_mod3}
Let $p\geq 5$ be prime and let $r,$ $1\leq r\leq p-1,$ be chosen such that $24r+1$ is a quadratic nonresidue modulo $p.$  Then, for all $n\geq 0,$ 
$$
f^{0}_{\hat{S},p}(pn+r)\equiv 0\pmod{3}.
$$
\end{theorem}
\begin{proof}
Thanks to Proposition \ref{genfn_pdefectzero_S} as well as the work completed by Olsson \cite{O3} just prior to his Proposition 9.10, we know the following:  
\begin{eqnarray*}
\sum_{n=0}^\infty f^{0}_{\hat{S},p}(n)q^n 
&=&
F^+_{\bar{p}}(q) + 2F^-_{\bar{p}}(q) \\
&\equiv &
F^+_{\bar{p}}(q) - F^-_{\bar{p}}(q) \pmod{3}\\
&=& 
(-q;-q)_\infty(-q^p;-q^p)_\infty^{t-1}.
\end{eqnarray*}
Given that we are only considering arithmetic progressions of the form $pn+r$ with $1\leq r\leq p-1,$ we can ignore the factor $(-q^p;-q^p)_\infty^{t-1}$ since it is a function of $q^p.$  Thus, we simply focus our attention on $(-q;-q)_\infty.$  

Next, note that $(-q;-q)_\infty$ is the product side of Euler's Pentagonal Number Theorem with $q$ replaced by $-q.$  Thus, we simply need to ask whether it is possible to write $pn+r$ as a pentagonal number.  The proof then follows exactly as the proof of Corollary \ref{spchar_hatS_mod3_infinitefamily} above.  
\end{proof}
\begin{example} Let $n$=31 and $p$=7. By Example \ref{yes}, there are two $\bar{7}$-cores of 31 with positive sign: $\lambda=(16,9,4,2)$ and $\lambda=(17,10,3,1)$; each label a self-associate 7-defect zero spin character of $\hat{S}(31)$. There are two $\bar{7}$-cores of $31$ with negative sign; each correspond to a pair of non-self-associate 7-defect zero spin characters of $\hat{S}(31)$. Hence there are six 7-defect zero spin characters of $\hat{S}(31)$ and $f^0_{\hat{S}}(31)\equiv 0\pmod{3}.$
\end{example}
\begin{corollary} Let $p\geq 5$ be prime and choose $r$, $1\leq r\leq p-1$, such that $24r+1$ is a quadratic nonresidue modulo $p$. Let $f^+_{\hat{S},p}(n)$ be the number of positive $p$-defect spin characters of $\hat{S}(n).$ Then, for all $n\geq 0$, 
\[
f^+_{\hat{S},p}(pn+r)\equiv 0\pmod{3}.
\] 
\end{corollary}
\begin{proof} The result follows from Corollary \ref{spchar_hatS_mod3_infinitefamily} and Theorem \ref{pdefectzero_S_mod3}.
\end{proof}
\subsection{$p$-defect zero spin characters of $\hat{A}(n)$}
Using Proposition \ref{genfn_pdefectzero_A}, we can now prove that $f^{0}_{\hat{A},p}$ satisfies infinitely many congruences modulo 3.  (This result looks extremely similar to Theorem \ref{pdefectzero_S_mod3} above.)  
\begin{theorem}
\label{pdefectzero_A_mod3}
Let $p\geq 5$ be prime and let $r,$ $1\leq r\leq p-1,$ be chosen such that $24r+1$ is a quadratic nonresidue modulo $p.$  Then, for all $n\geq 0,$ 
$$
f^{0}_{\hat{A},p}(pn+r)\equiv 0\pmod{3}.
$$
\end{theorem}
\begin{proof}
Thanks to Proposition \ref{genfn_pdefectzero_A} as well as the work completed above, we know the following:  
\begin{eqnarray*}
\sum_{n=0}^\infty f^{0}_{\hat{A},p}(n)q^n 
&=&
2F^+_{\bar{p}}(q) + F^-_{\bar{p}}(q) \\
&\equiv &
-F^+_{\bar{p}}(q) + F^-_{\bar{p}}(q)  \pmod{3} \\
&=& 
-(F^+_{\bar{p}}(q) - F^-_{\bar{p}}(q)) \\
&\equiv & 
- \sum_{n=0}^\infty f^{0}_{\hat{S},p}(n)q^n \pmod{3}.  
\end{eqnarray*}
This proof then follows from the proof of Theorem \ref{pdefectzero_S_mod3}. 
\end{proof}
\begin{example} Let $n=31$ and $p=7.$ We use Example \ref{yes}. Each of two self-associate spin characters of $\hat{S}(31)$ split into two distinct and conjugate 7-defect zero spin characters of $\hat{A}(31)$.   Each pair of non-self-associate spin characters of $\hat{S}(31)$ restrict to a single spin character of $\hat{A}(31).$ Hence there are six 7-defect zero spin characters of $\hat{A}(31)$ and $f^0_{\hat{A}}(31)\equiv 0 \pmod{3}.$
\end{example}
\begin{corollary} Let $p\geq 5$ be prime and choose $r$, $1\leq r\leq p-1$, such that $24r+1$ is a quadratic nonresidue modulo $p$. Let $f^+_{\hat{A},p}(n)$ be the number of positive $p$-defect spin characters of $\hat{A}(n).$ Then, for all $n\geq 0$
\[
f^+_{\hat{A},p}(pn+r)\equiv 0\pmod{3}.
\] 
\end{corollary}
\begin{proof}
The result follows from Corollary \ref{spchar_hatA_mod3_infinitefamily} and Theorem \ref{pdefectzero_A_mod3}.
\end{proof}

\section{Closing Thoughts}
It is intriguing to see that so many of the enumerating functions considered above satisfy infinitely many Ramanujan--like congruences modulo 2 and 3.  And while the proofs given above are satisfying (given their reliance on elementary generating function manipulations and classical product--to--sum results), it may prove profitable to prove these results from a combinatorial perspective.  Such combinatorial proofs can often provide additional insights into the arithmetic behavior of such functions.  

\section{Acknowledgements}
The first author thanks George Andrews for supporting his visit to Pennsylvania State University in November 2014 where this research began. He also thanks Henry Cohn and Massachusetts Institute of Technology as a visitor in March and April 2015 where he worked on the manuscript.  The first author was supported by grant PSC-CUNY TRADA-46-493 and a 2014-2015 Faculty Fellowship Leave from York College, City University of New York. Both authors would like to thank Jorn Olsson for his helpful comments and suggestions.

\end{document}